\documentclass[a4, 10pt,leqno]{amsart}
\usepackage{pifont}
\usepackage{mathtools}
\usepackage{mathabx}
\usepackage{bbm}
\usepackage{amsthm}
\usepackage{mathtools}
\usepackage{mathtools,accents}
\usepackage{amssymb,amscd,amsthm,enumerate}
\usepackage{mathrsfs}
\usepackage{a4wide}
\usepackage{eqnarray}
\usepackage{enumitem}
\usepackage[utf8]{inputenc}
\usepackage[T1]{fontenc}
\usepackage{lmodern}
\usepackage[english]{babel}
\usepackage{microtype}
\usepackage{float}
\usepackage{esint}
\usepackage{physics}
\usepackage{a4wide}
\usepackage{algorithmic}
\usepackage[ruled]{algorithm2e}
\usepackage[active]{srcltx}
\usepackage{lipsum}
\usepackage{upgreek}
\usepackage{listings}
\usepackage{wasysym}
\usepackage{csquotes}
\usepackage{enumerate}
\usepackage{color}
\usepackage[x11names,svgnames,usenames,dvipsnames]{xcolor}
\definecolor{ultrablue}{rgb}{0.0,0.0, 1}
\definecolor{jigari}{rgb}{0.39,0.0, 0.0}
\usepackage{url}
\usepackage{hyperref}
\hypersetup{
	linktoc=page,
	colorlinks,
	linkcolor={ultrablue},
	citecolor={ultrablue},
	urlcolor={ultrablue}
}
\hypersetup{breaklinks=true}
\usepackage{soul}
\usepackage{graphicx}
\usepackage{tikz}
\usepackage{tikzpagenodes}
\usepackage{scrlayer-scrpage}
\usepackage{tikz}
\usetikzlibrary{calc,spy}
\usetikzlibrary{shapes.geometric, arrows}
\usepackage{pgfplots}
\usetikzlibrary{intersections, pgfplots.fillbetween}
\usepackage{tikz}
\usetikzlibrary{matrix}
\usepackage[all]{xy}
\usepackage{subfig}
\newlength{\myeqskip}  
\AtBeginDocument{%
	\setlength\abovedisplayskip{\myeqskip}%
	\setlength\belowdisplayskip{\myeqskip}%
	\setlength\abovedisplayshortskip{\myeqskip-\baselineskip}%
	\setlength\belowdisplayshortskip{\myeqskip}}
\allowdisplaybreaks
\numberwithin{equation}{section}
\usepackage[a4paper,bindingoffset=0in,%
left=1.2in,
right=1.2in,
top=1.8in,
bottom=1.8 in,%
footskip=1mm]{geometry}
\theoremstyle{plain}
\newtheorem*{theorem*}{Theorem}
\newtheorem*{corollary*}{Corollary}
\newtheorem{lemma}{Lemma}[section]

\newtheorem{proposition}[lemma]{Proposition}

\theoremstyle{definition}

\newtheorem{remark}[lemma]{Remark}

\theoremstyle{remark}

\newcounter{example}

\makeatletter
\def\@xnamedef#1{\expandafter\protected@xdef\csname #1\endcsname}
\def\no@harm{} 
\def\ead@au#1{\protected@edef\@ead@au{#1}}
\patchcmd\runningauthor@fmt{\global\edef}{\protected@xdef}{}{}
\patchcmd\runningauthor@fmt{\global\edef}{\protected@xdef}{}{}
\patchcmd\author@fmt{\edef}{\protected@edef}{}{}
\patchcmd\add@xtok{\xdef}{\protected@xdef}{}{}
\makeatother
\newcommand{\secref}[1]{\textsection\thinspace\ref{#1}}

\newcommand{\R}{\mathbb R}
\newcommand{\Sp}{\mathbb S}

\DeclareMathOperator{\Hess}{Hess}

\DeclareMathOperator{\Riem}{Riem}

\newcommand{\K}{\mathcal K}
\newcommand{\y}{{\footnotesize \textbf{\textit{y}}}}
\newcommand{\vecv}{\footnotesize \textbf{\textit{v}}}
\newcommand{\vecw}{{\footnotesize \textbf{\textit{w}}}}

\newcommand{\dist}{\mathrm{d}}

\newcommand{\cpt}{{\textsf{CPT}}}
\usepackage{accents}
\newcommand*{\dt}[1]{%
		\accentset{\mbox{\bfseries .}}{#1}}

\makeatletter
\newcommand*\bigcdot{\mathpalette\bigcdot@{.5}}
\newcommand*\bigcdot@[2]{\mathbin{\vcenter{\hbox{\scalebox{#2}{$\m@th#1\bullet$}}}}}
\makeatother
\newcount\stylenum  \newdimen\styledim 
\def\varstyle#1{\mathchoice{\stylenum=0 #1}{\stylenum=1 #1}{\stylenum=2 #1}{\stylenum=3 #1}}
\def\mathaxis{\fontdimen22\ifcase\stylenum 
	\textfont\or\textfont\or\scriptfont\or\scriptscriptfont\fi2 }
\def\setstyledim{\styledim=\ifcase\stylenum .1em\or.1em\or.07em\or.05em\fi\relax}
\def\sqdot{\mathbin{\varstyle{\raise\mathaxis\hbox{\setstyledim
				\kern\styledim 
				\vrule width1.2\styledim height.6\styledim depth.6\styledim
				\kern\styledim}}}}

		\DeclareMathAlphabet{\mathdutchcal}{U}{dutchcal}{m}{n}
		\SetMathAlphabet{\mathdutchcal}{bold}{U}{dutchcal}{b}{n}
		\DeclareMathAlphabet{\mathdutchbcal}{U}{dutchcal}{b}{n}
		\DeclareSymbolFont{myletters}{OML}{ztmcm}{m}{it}
		\DeclareMathSymbol{\nicelambda}{\mathord}{myletters}{"15}
		\usepackage{nicefrac}
\usepackage{stackengine}
\setstackEOL{\\}
\newcounter{tmpctr}
\newcommand\fancyRoman[1]{%
	\setcounter{tmpctr}{#1}%
	\setbox0=\hbox{\kern.2pt\textsf{\Roman{tmpctr}}}%
	\setstackgap{S}{-.6pt}%
	\Shortstack{\rule{\dimexpr\wd0+.1ex}{.7pt}\\\copy0\\
		\rule{\dimexpr\wd0+.1ex}{.7pt}}%
}

		\usepackage{pifont}
		
\newcommand{\restr}{\raisebox{-.1908ex}{$\big|$}}
\makeatother
\newcommand{\ident}{\raisebox{0pt}{\scalebox{1.1}{$\mathbbm{1}$}}\hspace{-1pt}}

\makeatletter
\newcommand{\tpitchfork}{%
	\vbox{
		\baselineskip\z@skip
		\lineskip-.52ex
		\lineskiplimit\maxdimen
		\m@th
		\ialign{##\crcr\hidewidth\smash{$-$}\hidewidth\crcr$\pitchfork$\crcr}
	}%
}
\makeatletter
\newcommand{\tpmod}[1]{{\@displayfalse\pmod{#1}}}
\makeatother
\newcommand{\swap}{
\substack{\curvearrowright\\
	\curvearrowbotleft}}
\usepackage{stackengine,scalerel}

\newcommand\overstar[1]{\ThisStyle{\ensurestackMath{%
			\setbox0=\hbox{$\SavedStyle#1$}%
			\stackengine{0pt}{\copy0}{\kern0\ht0\smash{\SavedStyle\star}}{O}{c}{F}{T}{S}}}}

\newcommand\dunderline[3][-1pt]{{%
		\sbox0{#3}%
		\ooalign{\copy0\cr\rule[\dimexpr#1-#2\relax]{\wd0}{#2}}}}
\makeatletter
\@namedef{subjclassname@2020}{%
	\textup{2020} Mathematics Subject Classification}
\makeatother
\makeatletter 
\def\l@subsection{\@tocline{2}{0pt}{3pc}{6pc}{}}
\makeatother
\makeatletter 
\def\l@subsection{\@tocline{2}{0pt}{3pc}{6pc}{}}
\makeatother

\def\bysame{\leavevmode\hbox to3em{\hrulefill}\thinspace}
\begin{document}
\title[\scriptsize {On the rigidity of Finslerian conformal circle-preserving transformations}]{\small On the rigidity of \\ Finslerian conformal circle-preserving transformations} 
%
\author[\protect \scriptsize Z. Fathi]{Zohreh Fathi}
\author[\protect \scriptsize S. Lakzian]{Sajjad Lakzian$^{^{\scalebox{1}{$\star$}}}$}
\address{\noindent -- Z. Fathi \& S. Lakzian \newline  School of Mathematics, \newline Institute for Research in Fundamental Sciences (IPM), \newline P. O. Box 19395-
	5746, Tehran, Iran}
\address{\noindent -- S. Lakzian \newline \noindent Department of Mathematical Sciences\newline Isfahan University of  Technology (IUT) \newline Isfahan 8415683111, Iran.}
\email{\href{mailto:slakzian@iut.ac.ir}{slakzian@iut.ac.ir}}
\subjclass[2020]{53C60; 53C24}
\keywords{Finslerian manifold, conformal diffeomorphism, circle-preserving diffeomorphism, geodesically complete, constant flag curvature}
\thanks{ZF is partially supported by IPM grant No. 1403530043 and SL is partially supported by IPM grant No. 1405530313.}
\thanks{SL is supported by the INSF Grant No. 4030556, awarded by the “On
	the Frontiers of Mathematical Sciences” program.}
\thanks{$\raisebox{2pt}{$\star$}$\textit{the corresponding author}}
\maketitle
%
\begin{abstract}
\par \textsl{We prove that if a forward or backward complete Berwaldian or reversible Finslerian manifold $(M,F)$ admits a non-trivial (non-homothety) conformal concircular transformation ({\bf c}ircle-{\bf p}reserving {\bf t}ransformation or {\cpt} for short), $e^{\sigma} F$, where $\sigma$ has at least one critical point, then, $(M,F)$ is Riemannian. Consequently, $(M,g)$ is conformally diffeomoprhic to either 1) the standard sphere, 2) the Euclidean space, or 3) the hyperbolic space.}
\par \textsl{In particular, a compact Berwaldian or reversible Finslerian manifold does not admit any non-trivial conformal {\cpt}s, unless it is conformally diffeomorphic to the standard sphere.} 
\end{abstract}
\date{\today}
\section{Introduction}\label{sec:intro}
\par A geodesic circle in a Riemannian manifold $\left(M^n,g\right)$ is a smooth curve with constant first geodesic curvature $\kappa \ge 0$ and vanishing second geodesic curvature. Using a Frenet-Serret-Jordan apparatus, this means, once parametrized in arc-length, $X = c'$ and $Y = c''$ solve the ODE system 
\begin{align*}
	{X}'= Y, \quad {Y}'=-\kappa^2 X, \quad ('\; := \nabla_{{c}'});
\end{align*}
e.g. see~\cite{NM}.
\par A diffeomorphism, $\varphi$, between two Riemannian manifolds (or domains within) that preserves geodesic circles is referred to as a circle-preserving (a.k.a. concircular) transformation ({\cpt} for short). The initial development of the Riemannian theory is due to Yano~\cite{Yano1}--\cite{Yano5} with important contributions of~\cite{Ish-Tash,Ish, Tash1, Tash2, VO, Kan, Kuh} later on.
\subsection*{{\cpt}s in the Riemannian setting}
\par Vogel's theorem ensures that a Riemannian {\cpt} must be conformal, i.e., the pullback metric by $\varphi$ is $\widebar{g} = \rho^{-2} g$ and the inverse conformal factor $\rho$ satisfies the local PDE system $\rho_{i;j} = \lambda g_{ij}$ (semicolon $;$ stands for covariant differentiation) for a scalar function $\lambda$~\cite{Br,VO}. When $\varphi$ is a homothety, i.e., when $\lambda$ is constant, $\varphi$ is called a trivial {\cpt}. 
\par Letting $\sigma$ be the logarithmic conformal factor $\sigma = - \ln \rho$, the coordinate-free geometric PDE that characterizes a Riemannian {\cpt} is given by
\begin{align}\label{eq:pde-sigma}
	\nabla^2 \sigma - d\sigma\otimes d\sigma =\nicefrac{1}{n} \left(\Delta \sigma - \|\nabla \sigma\|^2\right) g;
\end{align}
see~\cite{Yano2, Kuh}.
\par It is important to note that, this is also the PDE that describes Einstein conformal transformations of an Einstein metric $g$; see~\cite{Br}. 
\par The classification of the global solutions of \eqref{eq:pde-sigma} - when $g$ is complete - goes as follows; see the collected works~\cite{Tash1,Tash2,Kan} and the finishing touch~\cite{Kuh}.
\par In the complete case, $\sigma$ can have at most two critical points and 
\begin{itemize}
	\item If $M^n$ is closed, then $(M^n,g)$ is conformally diffeomorphic to the standard sphere; furthermore, if $g$ is of constant scalar curvature, then $(M^n,g)$ is homothetic to the standard sphere. In this case, $\sigma$ has precisely two critical points. 
	\smallskip
	\item If $M$ is complete and non-compact and $\sigma$ has only one critical point, then $(M^n,g)$ is conformally diffeomorphic to either the flat Euclidean space or the hyperbolic space.
	\smallskip
	\item If $M$ is complete and non-compact and $\sigma$ has no critical points, then $(M^n,g)$ is conformally diffeomorphic to an isometric product $\R \times \Sigma$ in which $\Sigma$ is complete. 
\end{itemize} 
\subsection*{{\cpt}s in the Finslerian setting}
\par A Finslerian manifold $\left(M^n,F\right)$ is a direction-dependent generalization of the Riemannian one by replacing the Riemannian first fundamental form by a norm $F$ with certain smoothness and convexity properties; e.g., see \cite{BCS}. This norm gives rise to a Finslerian metric $g(x,\y)$ (or $g^{\y}$ for brevity), which depends on the direction $\y \neq 0$ as well.  
\subsubsection*{\small \sf \textbf{Finslerian geodesic circles}}
\par Let $c(t)$ be a smooth curve. We denote its natural lift to the tangent bundle by $\widetilde{c}(t) := \left(c(t), \dt{c}(t) \right)$. Let $\nabla$ denote the Cartan connection; see~\secref{subsec:cartan}.
\par The higher derivatives of $c$ can now be defined as $c^{(k)}(t) := \nabla^{k-1}_{\dt{\widetilde{c}}} \dt{c}$. For the covariant derivatives of $c$ w.r.t. a general parameter and the arc-length parameter (w.r.t. $g_{\dt{c}}$), we will adopt the standard notation $^{\dt{}}:= \nabla_{\dt{\widetilde{c}}(t)}$ and $':= \nabla_{{\widetilde{c}}'(s)}$ respectively. Though, for better readability, all derivatives of the coordinate functions will always be dotted.
\par As a result, a geodesic circle is characterized by the ODE 
$$
c'''(s) + \|c''(s)\|^2c'(s) = 0, \quad \text{or equivalently,} \quad c'''(s) + \kappa^2 c'(s) = 0.
$$
Furthermore, there is a unique geodesic circle parametrized by arc-length with the initial conditions $c(0) = p$, $c'(0) = u$, $c''(0) = v$ when $g_{u}(u,v) = 0$ and $\|u\| = 1$; in this case, $\kappa = \|c''(s)\| = \|v\|$.
For more detials, see~\cite{SY}.
\subsubsection*{\small \sf \textbf{Finslerian {\cpt}s}}
\par A Finslerian {\cpt} is a diffeomorphism $\varphi$ between two Finslerian structures that preserves Finslerian geodesic circles, i.e., takes geodesic circles to geodesic circles. 
\par \emph{Note that the inverse image of a geodesic circle is not necessarily a geodesic circle, i.e., {\cpt}s do not necessarily constitute a subgroup of the diffeomorphism group.} 
\subsubsection*{\small \sf \textbf{Conformality is not automatic}}
\par It was shown in~\cite{VO} that a Riemannian {\cpt} is automatically conformal. It is important to note that unlike the Riemannian setting, Finslerian {\cpt}s are not necessarily conformal, contrary to what is claimed in some literature. Indeed, not only does Vogel's proof break down, but also examples of circle-preserving fields (consequently maps) that are not conformal and vice versa are produced in~\cite[Section 6]{SY}; see also \cite{FL} for an explanation of the counterexamples. 
\par A straightforward modification of the proof of conformality given in~\cite{VO} gives something weaker than conformality; i.e., such maps preserve perpendicularity to the flagpole. We omit the proof. 
\begin{proposition}
	The perpendicularity ``$u \perp \y \;\; \text{w.r.t.} \;\; {g^{\y}}$'' is preserved under a Finslerian {\cpt}.
\end{proposition}
\begin{remark}
	To ensure conformality, all orthogonalities must be preserved. A natural sufficient condition for this to happen is if
	\begin{align*}
		g^{v} = \rho(v,w) g^{w}, \quad \text{and} \quad \widebar{g}^{v} = \widebar{\rho}(v,w) \widebar{g}^{w},
	\end{align*}
	holds for all $v$ and $w$ for some smooth functions $\rho$ and $\bar{\rho}$. But by keeping $w$ fixed, this implies $g^{\y}$ and $\widebar{g}^{\y}$ are conformal to Riemannian metrics and consequently must be Riemannian. This hints that ``Finslerian conformal {\cpt}s must perhaps fall not too far from Riemannian ones''. 
\end{remark}
\par \emph{The analysis in this work is limited to conformal {\cpt}s only. It is a classical fact that in a conformal transformation $\bar{F} = e^{\sigma} F$, $\sigma$ must be a scalar function}; see~\cite{Knebel}.

\par \dunderline{1.2pt}{\small \textit{\textbf{\textsf{Convention:}}}} \textsl{In what follows, we omit the reference to the diffeomorphism $\varphi$ and write a conformal {\cpt} as $\bar{F} = e^{\sigma} F$. We set $\rho = e^{-\sigma}$; we usually work with the PDE in terms of $\sigma$ but in cases where it is more convenient, we write some quantities in terms of $\rho$.} 
\par Recent contributions to the Finslerian theory of {\cpt}s include the important work~\cite{SY} in which a local characterization is established. Also in the recent work \cite{FL}, the authors have developed the theory further and in particular proved a Berwaldian rigidity. For more efficiency in our presentation, in many places, we will refer to \cite{FL}. 

\subsubsection*{\small \sf \textbf{The charcaterizing goemetric PDE}}
\par The coordinate-free PDE that describes a Finslerian conformal {\cpt} is presented in \cite{FL}. A Finslerian conformal {\cpt} satisfies
\begin{align}\label{eq:Fins-geom-pde}
	\nabla_{X^h} (\grad \sigma) -  d\sigma(X)\grad\sigma = f\left(x\right)X, \quad \forall X \in T_xM,
\end{align}
where $\nabla$ is the Cartan (or alternatively Chern) connection. Here, $X^h$ means horizontal lift of $X$; see~\secref{subsec:first-ord-calc}. $\grad \sigma$ is the gradient defined via the Legendre transform; see~\cite[Chapter 3]{Shen-lectures} and the discussion in~\cite[Section 2.1.1]{FL}.
\par Equivalently,
\begin{align*}
	\Hess_F \sigma - d\sigma\otimes d\sigma \circ \ident\times\ident = f\ident, \quad \text{on} \quad SM,
\end{align*} 
where $\Hess_F$ is the Finslerian Hessian; see~\cite[Chapter 14]{Shen-lectures}. Another equivalent formulation is
\begin{align*}
	{^{\sf horiz}\nabla}^2 \sigma -  d\sigma^2 = fg, 
\end{align*}
where ${^{\sf horiz}\nabla}$ means horizontal part of the Cartan (or Chern) connection. 
\par The salar function $f$ on the RHS of \eqref{eq:Fins-geom-pde} has the geometric formulation
\begin{align*}
	f(x) = \sigma''\restr_{\nicefrac{\y}{F(\y)}} - F(\grad \sigma)^2 + e^{2\sigma}\widebar{\kappa}^2\\
\end{align*}
where, $\sigma''\restr_{\nicefrac{\y}{F(\y)}}$ here means second derivative in the direction of the unit speed geodesic with initial velocity $\nicefrac{\y}{F(\y)}$ (indeed the Hessian) and $\widebar{\kappa}$ is the curvature of this geodesic w.r.t. $\widebar{F}$; for more detials, see~\cite{FL}.
\begin{remark}
	The local version of this PDE had been established in~\cite{SY} as the PDE system
	\begin{align*}
		\rho_{_{i|j}} = \left(e^{-\sigma}\right)_{i|j} = \lambda g_{ij},
	\end{align*}
where ``${}_{|}$'' denotes horizontal covariant differentiation w.r.t. the Cartan or Chern connection.
\end{remark}
	\par The importance of the coordinate-free formulation as presented in \cite{FL} is mainly the explicit geometric description of the scalar function $f$ in terms of first curvatures of paths. Indeed, also we have
	\begin{align}\label{eq:curv-disc}
		f(x)=  \sigma'' + \left< c'', \grad^{c'}\sigma \right>_{c'} - \|\grad^{c'} \sigma\|_{c'}^2 + e^{2\sigma}\widebar{\kappa}^2 -  \kappa^2,
	\end{align}
	in which $c(s)$ is an arbitrary unit speed (w.r.t. $F$) geodesic circle; see \cite{FL} for details.
	\par \emph{Notice $e^{2\sigma}\bar{\kappa}^2$ would coincide with $\kappa^2$ under a homothety. Therefore, the curvature discrepancy $e^{2\sigma}\bar{\kappa}^2 - \kappa^2$ appearing in \eqref{eq:curv-disc} measures how far the transformation is from being a homothety along a geodesic circle $c$.}
	\par \emph{For instance, in~\cite{FL}, taking advantage of this geometric formulation of $f$, the authors showed only normal geodesics are mapped to geodesics (integral curves of $\grad \sigma$) and consequently, other geodesics must map to pure circles (geodesic circles with $\kappa>0$). }
\subsubsection*{\small \sf \textbf{Rigidity in the presence of critical points}}
\begin{theorem*}\label{thrm:main}
Suppose $(M,F)$ is a forward or backward complete Berwaldian, or reversible Finslerian manifold that admits a non-trivial conformal {\cpt}, $\bar{F} = e^{\sigma} F$. If the critical set of $\sigma$ is non-empty, then $(M,F)$ is a Riemannian manifold. 
\end{theorem*}
Therefore, the Riemannian theory (see the beginning of the this section) can be applied, providing:
\begin{corollary*}\label{cor:main}
	Suppose $(M,F)$ is a forward or backward complete Berwaldian, or reversible Finslerian manifold that admits a non-trivial conformal {\cpt}, $e^{\sigma}F$, where $\sigma$ has at least one critical point. Then, $\sigma$ has no more than 2 critical points and $(M,F)$ is conformally diffeomorphic to either 
\begin{itemize}
	 \item the standard sphere ($\sigma$ has 2 critical points),
	 \smallskip
	 \item [] or,
	 \smallskip
	 \item the Euclidean space ($\sigma$ has 1 critical point),
	  \item [] or, 
	  \smallskip
	 \item  the hyperbolic space ($\sigma$ has 1 critical point). 
\end{itemize}
\end{corollary*}
\section{Preliminaries}
\par In these notes, we are concerned with smooth manifolds $M^{n}$ equipped with a Finslerian fundamental function $F(x,{\y})$ ($(x,{\y}) \in T_xM$). To $F$, one assigns $\y$-dependent metric tensor ${g}_{ij}(x,\y)$ and the Cartan tensor $\mathbf{C}(x,\y)$ with coefficients $C_{ijk}(x,\y)$. 
\par \dunderline{1.2pt}{\small \textit{\textbf{\textsf{Terminology:}}}} \textsl{With a slight abuse of terminologies, both a Finslerian connection and its corresponding covariant differentiation $\nabla$ will be referred to as a connection.}
\subsection{First order calculus}\label{subsec:first-ord-calc}
\par The geodesic spray coefficients and the nonlinear connection coefficients are given by 
\begin{align*}
	G^s = \nicefrac{1}{4} \; g^{sl}\left( y^k\dt{\partial}_l\partial_k\left(F^2\right) - \partial_l\left(F^2\right)\right),\quad G^i_j := \dt{\partial}_j G^i.
\end{align*}
\par The vectors 
\begin{align*}
	\delta_i := \partial_i - G_i^r\dt{\partial}_r, \quad \text{and} \quad \dt{\partial}_j := \partial_{y^j},
\end{align*}
span the horizontal sub-bundle, $\mathcal{H}_{(x,\y)}$ of $TTM$, and the vertical sub-bundle, $\mathcal{V}_{(x,\y)}$ respectively. For a vector $V \in TTM$, $hV$ and $vV$ denote the horizontal and vertical parts, respectively. The horizontal lift of a tangent vector $X = X^i\partial_i$ is given by $X^h = X^i\delta_{i}$. 
\par It is standard that 
\begin{align*}
	[\delta_i,\delta_j] = {\frak R}^k_{\;ij} \dt{\partial}_k,
\end{align*}
where
\begin{align}\label{eq:torsion-coeff}
	{\frak R}^k_{\;ij} = \delta_j G^k_i - \delta_i G^k_j,
\end{align}
is the $R^1$-torsion vector; see~\cite[(2.3.2.5)]{AIM}. 
In particular,
$
h\left[ \delta_i, \delta_j \right] = 0
$
\par Consequently,
\begin{align}\label{eq:torsion}
	[X^h,Y^h] = [X,Y]^h + {\frak R}(X,Y), \quad \text{where} \quad  {\frak R}(X,Y) = X^iY^j{\frak R}^k_{\;ij}\partial_k.
\end{align}
\subsubsection{\small \sf \textbf{The Cartan connection}}\label{subsec:cartan}
The Cartan connection (covariant differentiation) is defined as 
\par \begin{align*}
	\nabla_{{X}} \partial_i := \left( \overstar{\Gamma}^i_{jk} dx^j + C_{ij}^k\delta y^j \right)(X)  \; \partial_k, \quad X \in T\left(TM_0\right);
\end{align*}
More explicitly,
$
\nabla_{\delta_j} \partial_i = \overstar{\Gamma}^k_{ij} \partial_k
$
and
$
\nabla_{\dt{\partial}_j} \partial_i = C_{ij}^k \partial_k
$
where,
\begin{align*}
	\overstar{\Gamma}^k_{ij} := \nicefrac{1}{2}\; g^{kl}\left( \delta_i g_{jl} + \delta_j g_{il} - \delta_l g_{ij} \right).
\end{align*}
\par By a straightforward computation, we have
\begin{align}\label{eq:Chris-symbs}
	\overstar{\Gamma}^k_{ij} - \Gamma^k_{ij} =  g^{kl}\left(C_{ijr}G^r_l - C_{ljr}G^r_i - C_{lir}G^r_j\right);
\end{align}
e.g. see~\cite[2.4.9]{BCS}.
\par Key features of the Cartan connection include metric compatibility i.e. $^{\sf horiz}\nabla g = {^{\sf vert}\nabla} g = 0$, $h$-symmetricity, $v$-symmetricity and the fact that it has 3 non-zero torsion tensors (out of 5); for more details on the Cartan connection, see e.g.~\cite[Section 2.4.3]{AIM} and~\cite[Chapter II]{AZbook}.
\subsubsection{\small \sf \textbf{The Chern connection}}
The Chern connection is given by
$
\overset{\sf Ch}{\nabla}_{\delta_j} \partial_i = \overstar{\Gamma}^k_{ij} \partial_k
$
and
$
\overset{\sf Ch}{\nabla}_{\dt{\partial}_j} \partial_i = 0.
$
key features of the Chern connection include complete torsion-free ness and almost metric compatibility; see~\cite[Chapter 2]{BCS}. 
\par Since horizontally, Chern and Cartan connections coincide, in computations that are done using the Cartan connection and using only horizontal symmetry (as long as metric compatibility is not used), can be turned into computations in terms of the Chern connection by formally substituting $0$ for $\mathfrak{R}$. 
\subsection{Second order calculus}
\subsubsection{\small \sf \textbf{$\pmb{hh}$-curvature tensor}}\label{subsubsec:hh-curv}
In these notes, we are concerned with the Riemannian-like curvature tensor (as opposed to purely Finslerian curvature tensors) of a Finslerian connection $\nabla$; this curvature tensor is known as the $hh$-curvature. For the Cartan connection, we denote the $hh$-curvature by $\Riem^{\y}$, thus
\begin{align*}
	\Riem^{\y}(X,Y)Z := {\Omega}\left(X^h, Y^h  \right)Z,
\end{align*}
where $\Omega$ denotes the usual curvature form of the bundle $\mathcal{V}(M)$ over $TM_0$ (the slit tangent bundle) w.r.t. the Cartan connection $\nabla$ (as a Finslerian connection on $\mathcal{V}(M)$); for details see~\cite{AZbook}. Similarly, one also defines the $hh$-curvature tensor $\mathcal{R}iem$ corresponding to the Chern connection. 
\par In coordinates, following the index convention in~\cite{BCS}, we set
\begin{align*}
	\Riem(\partial_j,\partial_k)\partial_h = \Riem^{\;\;i}_{h\;\;jk}\partial_i,
\end{align*}
and similarly for $\mathcal{R}iem$. 
(c.f. \cite{AIM} and \cite{BF}, where the latter is $\Riem^{\;\;i}_{h\;\;kj}\partial_i$; hence, a sign difference).
\par The curvature symmetries that merely rely on metric compatibility and horizontal symmetry will still hold true for the $hh$-curvature tensor of the Cartan connection; in particular,
\begin{align*}
	\left<{\Riem^{\y}}(X,Y)Z,W\right>_{\y} = - \left<{\Riem^{\y}}(X,Y)W,Z \right>_{\y}. 
\end{align*}
and the Bianchi identity holds as well. As a result, one also has
\begin{align}\label{eq:curv-op-self-adj}
	\left<{\Riem^{\y}}(X,Y)Z,W\right>_{\y} = \left<{\Riem^{\y}}(Z,W)X,Y\right>_{\y}. 
\end{align}
The two curvature tensors $\Riem$ and $\mathcal{R}iem$ are related via
\begin{align*}
	\Riem(X,Y)Z = \mathcal{R}iem(X,Y)Z + \mathcal{S}(X,Y)Z,
\end{align*}
where
\begin{align*}
	S(X,Y)Z =  -\nabla_{{\frak R}(X,Y)}Z,
\end{align*}
equivalently, in coordinates,
\begin{align*}
	\mathcal{S}(\partial_k,\partial_j)\partial_h = - C^i_{hr} {\frak R}^{r}_{\;\;kj}.
\end{align*}
\subsubsection{\small \sf \textbf{The flag curvature}}
\par A flag is a product $v\wedge w$ ($v,w$ are independent tangent vectors) for which $v$ is called the flagpole. The flag curvature $\K(v\wedge w)$ is the sectional curvature ${\mathrm sec}(v,w)$ w.r.t. the osculating metric $g^V$, where $V$ is a vector field that extends $v$; see~\cite[Theorem 5.12]{Oh}. $\K(v\wedge w)$ is given by
\begin{align*}
	\K(\vecv \wedge \vecw) = \frac{\left<\Riem^{\vecv}\left(\vecv, \vecw\right)\vecw, \vecv \right>_{\vecv}}{\|\vecv\|_{\vecv}^2\|\vecw\|^2_{\vecv} - \left< \vecv, \vecw \right>_{\vecv}^2},
\end{align*}
and the same formula holds for $\K(\vecv\wedge \vecw)$ in terms of $\mathcal{R}iem$ as well.
\par Alternatively, for $\y$-dependent orthonormal frame $\left\{e_i\right\}$, $\K(e_i\wedge e_j)$ are eigenvalues of $-\Riem$ as a symmetric operator on $\bigwedge^2 \mathcal{V}M$. 
\subsubsection{\small \sf \textbf{Extrinsic and intrinsic geometry of hypersurfaces}}\label{subsubsec:hypersurface}
Let $\Sigma \subset M$ be a smooth hypersurface with normal ${\sf n}$. For $\y \in T_x\Sigma$, the normal bundle $\mathcal{V}_{(x,\y)}$ is a co-rank $1$ sub-bundle of the vertical bundle $\widetilde{\mathcal{V}}_{(x,\y)}$. Locally, a normal to $\Sigma$ naturally lifts to a complementary normal directed line bundle to $\mathcal{V}$, which can be taken to be the vertical lift, ${\sf n}^v$, of the normal ${\sf n}$. 
\par Let us denote the coordinates of the hypersurface by $u^2,\cdots,u^{n}$ signified by lower case Greek letter indices (as is customary). Denote the corresponding tangent coordinates by $v^\alpha$. Set $B^i_\alpha := \nicefrac{\partial x^i}{\partial u^\alpha}$. The Finslerian unit normal is then ${\sf n}^v = b^i\dt{\partial}_i$ that corresponds to the unit normal ${\sf n} = b^i{\partial}_i$ to the hypersurface. 
\par The ambient and induced double tangent frames are related via
\begin{align*}
	\dt{\partial}_{j} = {\mathcal{B}}^\alpha_j \, {^\Sigma\dt{\partial}}_\alpha + {\sf N}_j{\sf n}^v, \quad \text{and} \quad {^\Sigma\delta}_\alpha = B^i_\alpha \delta_i + H_\alpha {\sf n}^v,
\end{align*}
where
\begin{align*}
	{\mathcal{B}}^\gamma_i = g_{ij} B^j_\beta (^\Sigma g)^{\beta \gamma}, \quad \text{and} \quad {\sf N}_i = g_{ij} {b}^j.
\end{align*}
\par The coefficients of the nonlinear connection induced on $\Sigma$ are given by
\begin{align*}
	{^\Sigma G}^\beta_\alpha = {\mathcal{B}}^\beta_i \left( \left(\partial_\gamma B^i_{\alpha}\right)v^\gamma  + B^j_\alpha G^i_j \right).
\end{align*}
The second and third Rund-Brown tensors that play an important role in determining the induced connection and extrinsic curvature of $\Sigma$ are defined by
\begin{align*}
	M_\alpha := {C}_{ijk}\mathsf{N}^i\mathsf{N}^jB^k_\alpha, \quad \text{and} \quad	M_{\alpha\beta} := C_{ijk}B^i_\alpha B^j_\beta \; b^k.
\end{align*}
\par For more details on the second fundamental form and the Gauss-Codazzi equations in the Finslerian setting, as well as the type of the induced connection on a hypersurface, we refer to the original source~\cite{Mat1}; also to~\cite[Chapters 2 and 5]{BF} for a more detailed exposition.
\section{Highlights of the geometry of the underlying space}\label{sec:highlights}
Suppose $(M^n,F)$ is a Finslerian manifold that admits a non-trivial conformal {\cpt}. In this section we briefly mention recent developments about the geometry of $(M,F)$ as have been obtained in~\cite{FL}. 
\par Throughout these notes, $\Sigma$ denotes a regular level hypersurface of $\sigma$, either a portion thereof or the whole. 
\par On the regular domain, the integral curves of $\grad \sigma$ (once parametrized by arc-length) are geodesics and we refer to them as normal geodesics; see~\cite[Theorem 3.4]{FL}.
\par Take $u^1 = u$ to be the arc-length parameter along the normal geodesic curves. Choose any coordinate system $u^2,\cdots, u^{n-1}$ on the regular hypersurface $\sigma^{-1}(c)$. Extend these to a neighborhood of the point by taking them to be constant along the normal unit speed geodesic curves. This provides us with a system of local coordinates $(u^1 = u,u^2, \cdots, u^n)$ where $u^2, \cdots, u^n$ form coordinates for $\Sigma_x$ and $u$ is the arc-length of the normal geodesics, i.e., the coordinate $u$ is the one that integrates 
\begin{align*}
	\partial_u := \nicefrac{\grad^{\y} \sigma}{F(\grad^{\y} \sigma)} = \nicefrac{\grad \sigma}{F(\grad \sigma)}.
\end{align*}
\par \emph{The flow of the vector field $\grad \sigma$ is a local diffeomorphism. For any $s_1 \le s_2$ with $\sigma^{-1}([s_1,s_2]) \cap {\sf Crit}_{\sigma} = \varnothing$, the flow of $\grad \sigma$ becomes a diffeomorphism between any $\sigma^{-1}(t_1)$ and $\sigma^{-1}(t_2)$ for $t_1,t_2 \in (s_1,s_2)$. This means any chosen local coordinate system $(u^2, \cdots, u^n)$ for an open domain $U \subset \Sigma_{t}$ at time $t \in (s_1,s_2)$ can be extended to a coordinate system $u^1=u, u^2, \cdots, u^n$ on the open domain $(s_1,s_2)\times \Sigma_t$}.
\par \dunderline{1.2pt}{\small \textit{\textbf{\textsf{Notation:}}}} \textsl{Just as in~\cite{FL}, throughout these notes, the small case Greek letters signify the indices $2,\cdots,n$.}.
\par Obviously, locally, the distance of $u$-levels is determined by the difference in the value of $u$. Clearly $u$-levels correspond to the level hypersurfaces of $\sigma$ ($\sigma$-levels). Therefore, locally, $\sigma$ is a function of $u$. So we write $\sigma(u)$. Clearly
$
g^{\sf n}(\partial_1, \partial_1) = 1
$
and
$
g^{\sf n}(\partial_1, \partial_{\alpha}) = 0
$
i.e. $\partial_u = \partial_1 = {\sf n}$. 
\par In these coordinates, the following hold true. 
\begin{enumerate}
	\item \label{item:highlight-0} $f$ takes the form
	\begin{align*}
		f = \sigma_{uu} - \sigma_u^2 = -\nicefrac{\rho_{uu}}{\rho}.
	\end{align*}
	\smallskip
	\item \label{item:highlight-1} $g_{1\alpha}(\y) \equiv 0$, $g_{_{11}} \equiv 1$.
	\smallskip
	\item \label{item:highlight-2} $F^2(\grad \sigma) = \left< \grad \sigma,\grad \sigma \right>_{\sf n}$ is locally constant along $\Sigma$.
	\smallskip
	\item \label{item:highlight-3} $\nabla_{\dt{\partial}_j} \grad^{c'}\sigma = 0$.
	\smallskip
	\item \label{item:highlight-4} $C_{111} = C_{11\alpha} = C_{1\alpha\beta} = 0$.
	\smallskip
	\item \label{item:highlight-5} $G^s(\sf n) = 0$ for all $1\le s \le n$. 
	\smallskip
	\item\label{item:highlight-6}  $\dt{\partial}_u G^i \equiv 0$ i.e. $G^i_1\equiv 0$.
	\smallskip
	\item \label{item:highlight-7}
	$\dt{\partial}_u G^i_j \equiv 0$.
	\smallskip
	\item \label{item:highlight-8}
	The symbols $\overstar{\Gamma}^i_{jk}$ satisfy
	\begin{align*}
		\overstar{\Gamma}^1_{11} = \overstar{\Gamma}^1_{1\alpha} = \overstar{\Gamma}^\alpha_{11} =  0, \quad 
		\overstar{\Gamma}^\alpha_{1 \beta} = - \left(\nicefrac{\rho''}{\rho'}\right) \delta^\alpha_\beta, \quad \text{and} \quad
		\overstar{\Gamma}^1_{\alpha \beta} = \left(\nicefrac{\rho''}{\rho'}\right)g_{\alpha\beta}.
	\end{align*}
	\smallskip
	\item \label{item:highlight-9}
	$
	\partial_u g_{\alpha \beta} = 2 \left(\nicefrac{f}{\|\grad \sigma\|}\right)g_{\alpha\beta}.
	$ 
\end{enumerate}
For more details and the proofs, see~\cite[Theorems 3.4 and 4.1]{FL}. 
\par \emph{Another important fact established in \cite[Theorem~4.2]{FL} is that the set of critical points of $\sigma$, ${\sf Crit}$, is a discrete set, i.e., critical points are isolated.} 
\section{Local decomposition of the regular domain}
\begin{lemma}\label{lem:metric-decomp}
	Suppose $\bar{F} = e^{\sigma} F$ is a non-trivial {\cpt}. Then, locally on the regular domain of $\sigma$, $g(x,\y)$ decomposes as
	\begin{align}\label{eq:metric-decomp}
		g(x,\y) = du^2 + (e^{-\sigma}\sigma_u)^2\; {g}_{\sf ideal}(x,\vecv),\quad \vecv = (y^1,\cdots,y^{n-1}),
	\end{align}
	for a $(u^1,y^1)$-independent Finslerian metric ${g}_{\sf ideal}$ on $\Sigma$.  In other words, the isometric product 
	\begin{align*}
		F^2 = F_1^2 + (e^{-\sigma}\sigma_u)^2 F_\Sigma^2,
	\end{align*}
	holds true where $F_1$ is the standard length on $I$ and $F_\Sigma$ is a constant Finslerian norm on $\Sigma$. Clearly, 
	\begin{align}\label{eq:g-level}
	g_{\sf level} := (e^{-\sigma}\sigma_u)^2\; {g}_{\sf ideal}, 	
	\end{align}
	must be the Finslerian metric induced by $F$ on $\Sigma$. 
	\par Furthermore, for any regular value $\sigma(s_0)$, the warped product decomposition \eqref{eq:metric-decomp} holds on the entire domain $(s_0-\epsilon, s_0 +\epsilon) \times \sigma^{-1}(s_0)$. 
\end{lemma}
\begin{proof}[\footnotesize \textbf{Proof}]
	By \cite[Theorem 3.4]{FL}, we know that
	\begin{align*}
		\partial_u g_{\alpha \beta} = 2 \left(\nicefrac{f}{\|\grad \sigma\|}\right)g_{\alpha\beta} = -2 \left( \nicefrac{\rho_{uu}}{\rho_u} \right)g_{\alpha\beta}.
	\end{align*}
	Hence,
	\begin{align}
		\partial_u \Big(\nicefrac{g_{\alpha\beta}}{\left(e^{-\sigma}\sigma_u\right)^2}\Big) &= \Big(\nicefrac{\left(2e^{-2\sigma}f\sigma_u + 2e^{-2\sigma}(\sigma_u)^3 - 2e^{-2\sigma}\sigma_u\sigma_{uu}\right)}{\left(e^{-4\sigma}(\sigma_u)^4\right)}\Big) \; g_{\alpha\beta} \notag \\
		&= \left( \nicefrac{\left(2f + 2\sigma_u^2 - 2\sigma_{uu}\right)}{\left(e^{-2\sigma}(\sigma_u)^3\right)}\right) \; g_{\alpha\beta}\notag\\
		&= 0,\notag
	\end{align}
	where the last equality follows from $f = \sigma_{uu} - \sigma_u^2$; see \cite[Theorem 3.4, item (2)]{FL}.
	\par From \cite[Theorem 3.4, item (5)]{FL}, we get
	\begin{align*}
		\dt{\partial}_1 g_{\alpha\beta} = C_{1\alpha\beta} \equiv 0.
	\end{align*}
	Therefore, $\nicefrac{g_{\alpha\beta}}{\left(e^{-\sigma}\sigma_u\right)^2}$ is independent of $u^1,y^1$ and we set 
	\begin{align*}
		g_{\sf ideal} := \nicefrac{g_{\alpha\beta}}{\left(e^{-\sigma}\sigma_u\right)^2}.
	\end{align*}
	The conclusion then follows from $g_{_{11}}\equiv 1$; see~\cite[Theorem 3.4, item (5a)]{FL}.
	\par The metric decomposition is really independent of what local coordinates $u^2,\cdots, u^n$ we choose around a point in a regular level hypersurface; hence, the decomposition \eqref{eq:metric-decomp} can be extended along $\Sigma_{s_0} = \sigma^{-1}(s_0)$ indefinitely and extended in $u$ as long as $\sigma$ does not meet a critical point. Note the warping function is $e^{-\sigma}\sigma_u = \rho_u$.
\end{proof}
\par \dunderline{1.2pt}{\small \textit{\textbf{\textsf{Terminology/Notation:}}}}
\begin{itemize}
	\item \textsl{In the $u^i$ coordinates, we set $\y = (y^1, \vecv)$ where $\vecv = (y^2, \cdots, y^n)$ is also identified with $(0,y^2,\cdots,y^n)$.}
	\smallskip
	\item \textsl{We call a $(u,y^1)$-independent Finslerian manifold $\left(\Sigma, g_{\sf ideal}  \right)$ as in ~\hyperref[lem:metric-decomp]{Lemma~\ref{lem:metric-decomp}}, an \underline{\textit{ideal}} level hypersurface}. 
	\smallskip
	\item \textsl{The intrinsic geometric objects of $\Sigma$ w.r.t. $g_{\sf ideal}$ are mostly signified by an index ``{\footnotesize \sf ideal}''}.
	\item  \textsl{The intrinsic horizontal lift operator for $\Sigma$, which will be denoted by $h_{\Sigma}$ and the intrinsic double tangent vectors ${^\Sigma\dt{\partial}}_\alpha$ and ${^\Sigma\delta}_\alpha$. The intrinsic connection is denoted by $^\Sigma\nabla$. In other cases, the index ``{\footnotesize \sf level}'' will be used to refer to geometric quantities w.r.t. ``$g_{\sf level}$''; see~\eqref{eq:g-level}.}
	\smallskip
	\item \textsl{Going forward, $X,Y,Z$ denote vector fields that are tangent to the level hypersurfaces $\Sigma$.}
\end{itemize}
\begin{lemma}[The induced connection and curvature on $\Sigma$]\label{lem:curvature}
	Suppose $\bar{F} = e^{\sigma} F$ is a non-trivial {\cpt}. Then the following hold true:
	\begin{enumerate}
		\item \label{item:curvature-1} The induced connection from the Cartan connection on the level hypersurfaces $\Sigma$ coincides with the Cartan connection of the induced Finslerian metric. Also, the same holds for the Chern connection.  
		\par In addition, for a vector $\y = (y^1,\vecv)$ with $\vecv \neq 0$, we get
		\begin{align*}
			{^{\sf ideal} G}^\beta_\alpha(\y) =	{^{\sf level} G}^\beta_\alpha(\vecv)  = G^\beta_\alpha(\vecv), \quad
			{^{\sf ideal} C}_{\alpha\beta\gamma}(\y) = {^{\sf level} C}_{\alpha\beta\gamma}(\vecv) = C_{\alpha\beta\gamma}(\vecv),
		\end{align*}
		and
		\begin{align}
			{^{\sf ideal}\overstar{\Gamma}}^{\alpha}_{\beta\gamma}(\y) ={^{\sf level}\overstar{\Gamma}}_{\beta\gamma}(\text{$\vecv$}) = \overstar{\Gamma}^\alpha_{\beta\gamma}(\vecv).  \notag
		\end{align}  
		\smallskip
		\item \label{item:curvature-2} Setting $\mathcal{L}:= \left(\nicefrac{\rho_{uu}}{\rho_u}\right)^2$ and $\K := - \nicefrac{\rho_{uuu}}{\rho_u} = \left(f -\; \nicefrac{df(\grad\sigma)}{\|\grad\sigma\|^2}\right)$, the following identities hold for the $hh$-curvature tensor and flag curvatures. 
		\begin{enumerate}
			\smallskip
			\item \label{item:curv-1}
			${\Riem}^{\y}(X,Y)Z = {\Riem}_{\sf level}^{\vecv}(X,Y)Z - \mathcal{L} \Big( \left< Y,Z\right>_{\!\vecv}X - \left<X,Z \right>_{\!\vecv}Y \Big)$.
			\smallskip
			\item \label{item:curv-2}
			$
			\Riem^{\y}(X,Y){\sf n} = 0
			$ (in particular $\y$-independent).
			\smallskip
			\item \label{item:curv-3}
			$
			\Riem^{\y}(X,{\sf n}){\sf n} = (f -\; \nicefrac{df(\grad\sigma)}{\|\grad\sigma\|^2})X
			$ (in particular $\y$-independent).
			\smallskip
			\item \label{item:curv-4}
			$
			\Riem^{\y}(X,{\sf n})Y =  -\K \left< X,Y \right>_{\!\vecv} {\sf n}
			$
			(in particular $y^1$-independent).
			\smallskip
			\item \label{item:curv-5}
			\smallskip
			Items (\ref{item:curv-1}) -- (\ref{item:curv-4}) also hold for $\mathcal{R}iem$ verbatim.
			\smallskip	
			\item \label{item:curv-6}
			$
			\K\left( \y \wedge Y \right) = {\K}_{\sf level}(\vecv \wedge Y) - \mathcal{L}
			$
			in particular 
			$
			\K\left( X \wedge Y \right) = {\K}_{\sf level}(X \wedge Y) - \mathcal{L}.
			$
			\smallskip
			\item \label{item:curv-7}
			$
			\K\left( {\sf n} \wedge X \right) = \K
			$.
		\end{enumerate}
		\smallskip
		\item  \label{item:curvature-3} $\dt{\partial}_u \Riem^{\;\;i}_{h\;\;jk} \equiv0$; consequently, $\Riem^{\vecv} = \Riem^{\y}$ when $\vecv \neq 0$. 
	\par Using (\ref{item:curv-1}) -- (\ref{item:curv-4}), in coordinates, one has
		\begin{align*}
			\Riem^{\;\;1}_{1\;\;\alpha\beta}(\y) = \Riem^{\;\,\eta}_{1\;\;\alpha\beta}(\y) = \Riem^{\;\;\alpha}_{\beta\;\;1\gamma}(\y) = \Riem^{\;\;1}_{1\;\;1\alpha}(\y) = \Riem^{\;\;1}_{1\;\;1\alpha}(\y) = 0,
		\end{align*}
		\begin{align*}
			\Riem^{\;\,\alpha}_{1\;\;1\beta}(\y) = \left( \nicefrac{\rho_{uuu}}{\rho_u} \right)\delta^\alpha_\beta, \quad \Riem^{\;\;1}_{\beta\;\;1\gamma}(\y) = \K g_{\beta \gamma}, \quad \text{and}
		\end{align*}
		\begin{align*}
			{\Riem}^{\;\;\alpha}_{\beta\;\;\gamma\kappa} (\y) = {^{\sf level} \Riem}^{\;\;\alpha}_{\beta\;\;\gamma\kappa} (\vecv) -\left(\nicefrac{\rho_{uu}}{\rho_u}\right)^2\left( g_{\beta \kappa}\delta^\alpha_\gamma - g_{\beta\gamma}\delta^\alpha_\kappa \right).
		\end{align*}
	\par	Other curvature coefficients are obtained by symmetries of the curvature tensor. 
		\par \underline{As a result of \hyperref[item:curv-5]{item (\ref{item:curv-5})}, the same local formulas hold for $\mathcal{R}iem$}.
	\end{enumerate}
\end{lemma}
\begin{remark}
	\textsl{One immediate consequence of ~\hyperref[lem:curvature]{Lemma~\ref{lem:curvature}} is that the properties of having isotropic, sectional or constant flag curvature are inherited from $(M,F)$ by the level hypersurfaces $\Sigma$.}
\end{remark}
\begin{proof}[\footnotesize \textbf{Proof}]
	\hfill
	\begin{enumerate}
		\item []{\bf{\footnotesize \textsf{Proof of (\ref{item:curvature-1})}}}.
	According to \secref{subsubsec:hypersurface}, we set 
	\begin{align*}
		B^i_\alpha := \nicefrac{\partial x^i}{\partial u^\alpha} = \delta^i_\alpha, \quad \text{and}, \quad b^i = \delta^i_{1}.
	\end{align*}
	Consequently, we observe that
	\begin{align*}
		{\mathcal{B}}^\gamma_\alpha = \delta^\gamma_{\alpha}, \quad {\mathcal{B}}^\gamma_1 = 0, \quad \text{and}, \quad {\sf N}_i = \delta_{i1}.
	\end{align*} 
	\par The ambient and induced double tangent frames are related via
	\begin{align*}
		\dt{\partial}_{j} = {\mathcal{B}}^\alpha_j \, {^\Sigma\dt{\partial}}_\alpha + {\sf N}_j{\sf n}^v, \quad \text{and} \quad {^\Sigma\delta}_\alpha = B^i_\alpha \delta_i + H_\alpha {\sf n}^v,
	\end{align*}
	where
	\begin{align*}
		H_\alpha = {\mathsf{N}}_i \left(\left(\partial_\gamma B^i_{\alpha}\right)v^\gamma+ B^j_\alpha {G}^i_j  \right) = \delta_{i1} B^j_\alpha {G}^i_j = {G}^1_\alpha.
	\end{align*}	
	Therefore, 
	\begin{align*}
		\dt{\partial}_{1} = {\sf n}^v, \quad {^\Sigma\dt{\partial}}_\alpha =\dt{\partial}_{\alpha}, \quad {^\Sigma\delta}_\alpha = \delta_\alpha + G^1_\alpha {\sf n}^v;
	\end{align*}
	see~\secref{subsubsec:hypersurface} for more details. 
	\par To show that the induced connection is Cartan, we compute the third Rund-Brown tensor $M_{\alpha\beta}$. Since $b^k = \delta^k_1$ and $B^i_\alpha = \delta^i_\alpha$, we deduce that
	\begin{align*}
		M_{\alpha\beta} = C_{ijk}B^i_\alpha B^j_\beta \; b^k = C_{\alpha\beta 1} \equiv 0,
	\end{align*}
	Thus, by \cite[Theorem 5.2]{Mat1}, we know that the induced connection is the Cartan connection of the induced metric. 
	\par By \cite[II.(1.21)]{BF}, it follows 
	\begin{align}\label{eq:G-match}
		{^{\sf level} G}^\beta_\alpha (\vecv) &= {\mathcal{B}}^\beta_i \left( \left(\partial_\gamma B^i_{\alpha}\right)v^\gamma  + B^j_\alpha G^i_j \right) \\
		&=\delta^\beta_{i}\delta^j_\alpha G^i_j \notag \\
		&= G^\beta_\alpha (\vecv)\notag.
	\end{align}
	Also, by \secref{sec:highlights}\,-\,\hyperref[item:highlight-4]{item (\ref{item:highlight-4})}, we obtain
	\begin{align}\label{eq:C-match}
		^{\sf level} C_{\alpha\beta\gamma}(\vecv) &= \nicefrac{1}{2} \, {^\Sigma\dt{\partial}_{\gamma}} {^\Sigma g}_{\alpha \beta} \\
		&= \nicefrac{1}{2} \, {^\Sigma\dt{\partial}_{\gamma}} {g}_{\alpha \beta} \notag \\
		&= \nicefrac{1}{2} (\dt{\partial}_{\alpha} -  {\sf N}_\alpha \dt{\partial}_1) {g}_{\alpha \beta} \notag\\
		&= {C}_{\alpha\beta\gamma} - \nicefrac{1}{2}\,{\sf N}_\alpha \dt{\partial}_1 (g_{\alpha \beta})\notag \\
		&= {C}_{\alpha\beta\gamma} - \,{\sf N}_\alpha C_{\alpha\beta 1}\notag\\
		& = C_{\alpha\beta\gamma}(\vecv)\notag.
	\end{align}
	Now 
	$
	\overstar{\Gamma}^\alpha_{\beta\gamma}(\text{$\vecv$}) = {^{\sf level}\overstar{\Gamma}}^{\alpha}_{\beta\gamma}(\vecv)
	$
	follows from \eqref{eq:G-match} and \eqref{eq:C-match} and the identity \eqref{eq:Chris-symbs} combined with the fact, ${\Gamma}^\alpha_{\beta\gamma} = {^{\sf level}\overstar{\Gamma}}^{\alpha}_{\beta\gamma}$ that is straightforward from the metric decomposition \eqref{eq:metric-decomp}.
	\par The conclusion follows as we note that $\dt{\partial}_u G^i_j \equiv 0$ and $\dt{\partial}_u C_{ijk} \equiv 0$ result in $\dt{\partial}_u \overstar{\Gamma}^{k}_{ij} = 0$ i.e. these quantities are independent of $y^1$; in particular in our setting, we can set $y^1 = 0$. 
	\par Since the Chern connection is indeed the horizontal Cartan covariant differentiation, it immediately follows that the induced connection from Chern connection is the Chern connection of the induced metric as well. 
	\smallskip	
	\item []{\bf{\footnotesize \textsf{Proof of (\ref{item:curvature-2})}}}.
	\emph{In what follows, $\nabla$ is the Cartan connection.} Let us derive the (Gauss-Peterson-Minardi-Codazzi) hypersurface curvature relations. We find it more illuminating to do this in a hybrid coordinate and coordinate-free manner than merely in coordinates.
	\par Let $X = X^\alpha\partial_\alpha$; thus, $X^h = X^\alpha\delta_\alpha$. Consequently
	\begin{align*}
		{X^{h_\Sigma}} &= X^\alpha \; {^\Sigma \delta_\alpha} \\
		&= X^\alpha \left( \delta_{\alpha} + {G}^1_\alpha \dt{\partial}_u \right) \\
		&= X^h + X^\alpha G^1_\alpha \dt{\partial}_u.
	\end{align*}
	Using the PDE \eqref{eq:Fins-geom-pde} and noting $\grad\rho = - \rho\grad \sigma$, we get
	\begin{align}\label{eq:induced-nabla-rho}
		\nabla_{X^{h_\Sigma}} {\grad \rho} &= \nabla_{\left( X^h + X^\alpha G^1_\alpha \dt{\partial}_u \right)}  {\grad \rho}\\
		&= -f\rho X + X^\alpha G^1_\alpha\nabla_{\dt{\partial}_u} \left({\rho_u \partial_u}\right)\notag \\
		&=-f\rho X + X^\alpha \rho_u G^1_\alpha  C^k_{11}\partial_k \notag \\
		&=-f\rho X. \notag
	\end{align}
	Similarly,
	\begin{align}\label{eq:tang-horiz-cov-der}
		\nabla_{X^{h_\Sigma}} {Y} &= \nabla_{\left( X^h + X^\alpha G^1_\alpha \dt{\partial}_u \right)} {Y}\\
		&= \nabla_{X^h} {Y} + X^\alpha G^1_\alpha \nabla_{\dt{\partial}_u} \left({Y^j\partial_j}\right)\notag\\
		&=  \nabla_{X^h} {Y} + X^\alpha Y^j G^1_\alpha C^k_{1j}\partial_k\notag\\
		&= \nabla_{X^h} {Y}.\notag
	\end{align}
	Using \eqref{eq:induced-nabla-rho} and noting ${\sf n} = \nicefrac{\grad\rho}{\rho_u}$, one obtains
	\begin{align*}
		\nabla_{X^{h_\Sigma}} Y &= {^\Sigma\nabla}_{X^{h_\Sigma}} Y - \left< \nabla_{X^{h_\Sigma}} {\sf n},Y \right> {\sf n}\\
		&= {^\Sigma\nabla}_{X^{h_\Sigma}} Y + \left(\nicefrac{f\rho}{\rho_{uu}}\right)\left<X,Y\right> {\sf n}\\
		&= {^\Sigma\nabla}_{X^{h_\Sigma}} Y - \left(\nicefrac{\rho_{uu}}{\rho_u^2}\right) \grad \rho.
	\end{align*}
	\begin{enumerate}
		\smallskip
		\item []{{\footnotesize \textsf{Proof of (\ref{item:curv-1})}}.} 
		Upon iteration, we get
		\begin{align*}
			\nabla_{X^{h_\Sigma}} \nabla_{Y^{h_\Sigma}} Z  &= {^\Sigma\nabla}_{X^{h_\Sigma}} {^\Sigma\nabla}_{Y^{h_\Sigma}} Z - \left(\nicefrac{\rho_{uu}}{\rho_u^2}\right)  \left<X,  {^\Sigma\nabla}_{^\Sigma Y^h} Z \right> \grad \rho\\
			&- \left(\nicefrac{\rho_{uu}}{\rho_u^2}\right) \left( \left<\nabla_{^\Sigma X^h} Y, Z \right> + \left< Y, \nabla_{^\Sigma X^h} Z \right>\right) \grad \rho\\
			& - \left(\nicefrac{\rho_{uu}}{\rho_u}\right)^2 \left<Y,Z \right> X.
		\end{align*}
		Let ${\frak R}$ denote the h-torsion tensor (anti-symmetric) and $^\Sigma {
			\frak R}$ the same tensor for $\Sigma$, by \eqref{eq:torsion}, we have
		\begin{align}\label{eq:torsion-R-1}
			[X^h,Y^h] = [X,Y]^h + {\frak R}(X,Y),
		\end{align}	
		and similarly,
		\begin{align*}
			[X^{{h_\Sigma}}, Y^{{h_\Sigma}}] = {[X,Y]^{h_\Sigma}} + {^\Sigma {\frak R}}(X,Y).
		\end{align*}
		Thus, as a reuslt of \eqref{eq:tang-horiz-cov-der}, 
		\begin{align*}
			\nabla_{[X^{{h_\Sigma}}, Y^{{h_\Sigma}}]} Z &= \nabla_{{[X,Y]^{h_\Sigma}}}Z + \nabla_{{^\Sigma {\frak R}}(X,Y) }Z\\
			& = \nabla_{[X,Y]^h}Z + \nabla_{{^\Sigma {\frak R}}(X,Y) }Z.
		\end{align*}
		holds true. 
		\par By the \emph{horizontal symmetry} of the induced Cartan connection and the standard facts about the Lie bracket
		\begin{align}\label{eq:tang-horiz-bracket}
			{^\Sigma\nabla}_{X^{{h_\Sigma}}}Y - {^\Sigma\nabla}_{Y^{{h_\Sigma}}} X = {^\Sigma[X,Y]} = [X,Y].
		\end{align}
		Combining \eqref{eq:tang-horiz-cov-der}, \eqref{eq:tang-horiz-bracket} and \eqref{eq:torsion-R-1}, we get
		\begin{align}\label{eq:omega-diff}
			&\Omega(X^h,Y^h)Z -	{^\Sigma\Omega}(X^{{h_\Sigma}},Y^{{h_\Sigma}})Z \\
			&= \nabla_{{\frak R}(X,Y)}Z - {^\Sigma\nabla}_{^\Sigma {\frak R}(X,Y)}Z - \left(\nicefrac{\rho_{uu}}{\rho_u}\right)^2\Big(\left<Y,Z \right> X  - \left< X,Z\right> Y\Big).\notag
		\end{align}
		\par {\small \bf Claim.} ${\frak R}(X,Y) = {^\Sigma{\frak R}}(X,Y)$ and $\nabla_{{\frak R}(X,Y)}Z = {^\Sigma\nabla}_{{^\Sigma{\frak R}}(X,Y)}Z$. 
		\par {\small \bf Proof of the claim:}
		By \eqref{eq:torsion-coeff} and \hyperref[item:curvature-1]{item (\ref{item:curvature-1})} above, in coordinates, we deduce that
		\begin{align}\label{eq:R-match}
			^\Sigma {\frak R}^\alpha_{\;\;\beta\gamma}(\vecv) &= ({\delta}_\beta + {G}^1_\beta \dt{\partial}_u)  \; {^{\sf level}G}^\alpha_\gamma - ({\delta}_\gamma + {G}^1_\gamma \dt{\partial}_u) \; {^{\sf level} G}^\alpha_\beta\\
			&= {\frak R}^\alpha_{\;\;\beta\gamma} (\vecv).\notag
		\end{align}
		where we have used $\dt{\partial}_u G^i_j \equiv 0$ from \secref{sec:highlights}\,-\,\hyperref[item:highlight-7]{item (\ref{item:highlight-7})}.
		\par Then, recalling ${\frak R}$ is vertical, by using \eqref{eq:C-match}, one gets
		\begin{align}\label{eq:vertical-dif}
			\nabla_{\dt{\partial}_\alpha}\partial_\beta - {^\Sigma\nabla}_{^\Sigma\dt{\partial}_\alpha}\partial_\beta &= C^i_{\alpha\beta}\partial_i - {^\Sigma C}_{\alpha\beta}^\gamma \partial_\gamma\\
			&= C^\eta_{\alpha\beta}\partial_\eta - {^\Sigma C}_{\alpha\beta}^\gamma \partial_\gamma\notag\\
			&= 0.\notag
		\end{align}
	The claim follows upon the combination \eqref{eq:R-match} and \eqref{eq:vertical-dif}. \scalebox{0.6}{$\blacksquare$}
\par Thus, noting \eqref{eq:omega-diff}, we have proven that
\begin{align*}
	\Riem^{\vecv}(X,Y)Z = {^\Sigma\Riem}^{\vecv}(X,Y)Z   - \left(\nicefrac{\rho_{uu}}{\rho_u}\right)^2\Big(\left<Y,Z \right>_{\vecv} X  - \left< X,Z\right>_{\vecv} Y\Big),
\end{align*}
which is, in particular, always tangent to $\Sigma$. The conclusion then follows from $\Riem^{\y} = \Riem^{\vecv}$ (when $\vecv \neq 0$) which will be established in the proof of \hyperref[item:curvature-3]{item (\ref{item:curvature-3})} below. 
\smallskip
\item	[]{{\footnotesize \textsf{Proof of (\ref{item:curv-2})}}.} 
Using \secref{sec:highlights}\,-\,\hyperref[item:highlight-3]{item (\ref{item:highlight-3})}, we deduce that
\begin{align*}
	\Riem(X,Y)\grad \sigma &=
	\Omega(X^h,Y^h)\grad \sigma\notag\\
	&= \nabla_{X^h}\nabla_{Y^h} \grad \sigma - \nabla_{Y^h}\nabla_{X^h} \grad \sigma - \nabla_{\left[X^h, Y^h \right]} \grad \sigma \notag\\
	&= \nabla_{X^h}\left( fY + (Y\sigma) \grad\sigma  \right) - \nabla_{Y^h}\left( fX + (X\sigma) \grad\sigma  \right) \\
	&- \nabla_{\left[X, Y \right]^h + {\mathfrak{R}}(X,Y)} \grad \sigma \\
	&= (Xf)Y - (Yf)X + f[X,Y] + \left([X,Y]\sigma\right)\grad \sigma\\
	&+ (Y\sigma)\left( fX + (X\sigma)\grad\sigma \right) - (X\sigma)\left( fY + (Y\sigma)\grad\sigma \right)\\
	& - f[X,Y] - ([X,Y]\sigma)\grad\sigma\\
	&= (Xf)Y - (Yf)X\\
	& = 0;
\end{align*}
i.e., $\Riem(X,Y){\sf n} = 0$.
\smallskip 	 
\item	[]{{\footnotesize \textsf{Proof of (\ref{item:curv-3})}}.}
Using the $h$-symmetry of the Cartan connection, we can write
\begin{align*}
	\Riem(X,\grad \sigma)\grad \sigma &= \nabla_{X^h}\nabla_{(\grad\sigma)^h} \grad \sigma - \nabla_{(\grad\sigma)^h}\nabla_{X^h} \grad \sigma - \nabla_{\left[X^h, (\grad\sigma)^h \right]} \grad \sigma \notag\\
	&= \nabla_{X^h}\left( f\grad\sigma + \|\grad\sigma\|^2 \grad\sigma  \right) - \nabla_{(\grad\sigma)^h}\left( fX + (X\sigma) \grad\sigma  \right) \\
	&- \nabla_{\left[X, \grad\sigma \right]^h} \grad \sigma - \nabla_{{\frak R}(X, \grad\sigma )} \grad \sigma \\
	&= (Xf)\grad\sigma + f\nabla_{X^h}\grad\sigma\\
	& + (X\|\grad\sigma\|^2)\grad\sigma + \|\grad\sigma\|^2(fX + (X\sigma)\grad\sigma)\\
	& - df(\grad\sigma)X - f \nabla_{(\grad\sigma)^h} X\\
	& - f[X,\grad\sigma] - ([X,\grad\sigma]\sigma)\grad\sigma\\
	&= \Big(\|\grad\sigma\|^2f - df(\grad\sigma)\Big)X,
\end{align*}
where, we have used \secref{sec:highlights}\,-\,\hyperref[item:highlight-3]{item (\ref{item:highlight-3})}. In the last equality, we have used the h-symmetry
\begin{align*}
	\nabla_{X^h}\grad\sigma - \nabla_{(\grad\sigma)^h} X = [X,\grad\sigma],
\end{align*}
along with 
\begin{align*}
	Xf = X\sigma = X\|\sigma\|^2 = 0,
\end{align*}	
that implies 
\begin{align*}
	[X,\grad\sigma]\sigma = X(\|\grad \sigma\|^2) - \grad\sigma X(\sigma) = 0.
\end{align*}
In particular, one obtains
\begin{align*}
	\Riem(X,{\textsf{n}}){\textsf{n}} = \Big(f - \; \nicefrac{df(\grad\sigma)}{\|\grad\sigma\|^2}\Big)X,
\end{align*}
\emph{which is $\y$-independent!}. 
\smallskip
\item	[]{{\footnotesize \textsf{Proof of (\ref{item:curv-4})}}.}
 By the symmetricity of the Cartan curvature operator (see~\eqref{eq:curv-op-self-adj}), we get
\begin{align*}
	\left<\Riem(X,\grad \sigma)Y, Z\right> = \left<\Riem(Y,Z)X,\grad\sigma\right> = 0.
\end{align*}
Thus, from \hyperref[item:curv-3]{item (\ref{item:curv-3})} above, one deduces that
\begin{align*}
	\left<\Riem(X,\grad \sigma)Y, \grad\sigma\right> &=  -\left<\Riem(Y,\grad\sigma)\grad\sigma, X\right> \\
	&= - \Big(\|\grad\sigma\|^2f - df(\grad\sigma)\Big)\left< X,Y \right> \\
	&= -\K \|\grad\sigma\|^2\left< X,Y \right>.
\end{align*}
Therefore,
\begin{align*}
	\Riem(X,\grad \sigma)Y =  -\K \left< X,Y \right> \grad\sigma,
\end{align*}
or
\begin{align*}
	\Riem(X,{\sf n})Y =  -\K \left< X,Y \right> {\sf n}.
\end{align*}
\smallskip
\item	[]{{\footnotesize \textsf{Proof of (\ref{item:curv-5})}}.}
Indeed as we saw, using the Chern connection, the proofs of \hyperref[item:curv-1]{items (\ref{item:curv-1})-- (\ref{item:curv-3})} are almost verbatim; indeed, since in their proofs, we have only used $h$-symmetricity of the Cartan connection (not in the proof of \hyperref[item:curv-4]{item (\ref{item:curv-4})} where symmetry of the curvature operator needs metric compatibility), substitute $0$ for the torsion ${\frak R}$ to obtain the argument for the Chern connection.
\par For \hyperref[item:curv-4]{item (\ref{item:curv-4})}, we can write
\begin{align*}
	\mathcal{R}iem(X,{\sf n})Y &= \Riem(X,{\sf n})Y + \nabla_{{\frak R}(X,{\sf n})}Y \\
	&= -\K \left< X,Y \right> {\sf n} + \nabla_{{\frak R}(X,{\sf n})}Y,
\end{align*}
and by \secref{sec:highlights}\,-\,\hyperref[item:highlight-4]{item (\ref{item:highlight-4})}, we deduce that $ \nabla_{{\frak R}(X,{\sf n})}Y =0$ because
\begin{align*}
	\nabla_{{\frak R}(\partial_i,{\sf n})}\partial_j = C^1_{j r} \mathfrak{R}^r_{\;\,i1} = C_{1jr}\mathfrak{R}^r_{\;\,i1} = 0.
\end{align*}
This yileds
\begin{align*}
	\mathcal{R}iem^{\;\;1}_{\beta\;\;1\gamma} = \K \delta_{\beta\gamma}.
\end{align*}
\smallskip
\item	[]{{\footnotesize \textsf{Proof of (\ref{item:curv-6})}}.}
 Straightforward from \hyperref[item:curv-1]{item (\ref{item:curv-1})} above. 
\smallskip
\item	[]{{\footnotesize \textsf{Proof of (\ref{item:curv-7})}}.}
Straightforward from \hyperref[item:curv-3]{item (\ref{item:curv-3})} above.	
\end{enumerate}
\smallskip
\item []{\bf{\footnotesize \textsf{Proof of (\ref{item:curvature-3})}}}.
 Recall 
\begin{align*}
\Riem^{\;\;i}_{h\;\,jk} = K^{\;\;i}_{h\;\;jk} - C^i_{hr}{\frak R}^r_{jk},
\end{align*}
where the tensor $K$ is 
\begin{align*}
K^{\;\;i}_{h\;\;jk} = \delta_k\overstar{\Gamma}^i_{hj} + \overstar{\Gamma}^r_{hj}\overstar{\Gamma}^i_{rk} - j\swap k,
\end{align*}
and the $R^1$ torsion is
\begin{align*}
{\frak R}^i_{jk} = \delta_k G^i_j - j\swap k.
\end{align*}
Here, $j\swap k$ denotes the term(s) obtained from the preceding term(s) by interchanging
the indices j and k. See~\cite[Section 2.3.2]{AIM} for details (Caution: there is a sign difference because of indexing conventions; explained in~\secref{subsubsec:hh-curv}).
\par Now $\dt{\partial}_u K^{\;\;i}_{h\;\;jk} = 0$ follows from $\dt{\partial}_u \overstar{\Gamma}^{i}_{jk} \equiv 0$ and $\dt{\partial}_u G^i_j \equiv 0$ while $\dt{\partial}_u {\frak R}^i_{jk} = 0$ also needs $\dt{\partial}_u C^{i}_{jk} \equiv 0$; all established in \cite{FL}; see also \secref{sec:highlights}.
\par Consequently, for $\y = (y^1,\vecv)$ with $\vecv \neq 0$, 
\begin{align*}
\Riem^{\vecv} = \Riem^{\y},
\end{align*}
holds true and the conclusion follows.	 
\end{enumerate}
\end{proof}
\section{Proof of the main result}
\subsection{The geometry around critical points}
\begin{lemma}\label{lem:level-const-curv}
	Suppose $(M,F)$ is backward or forwar dcomplete and $e^{\sigma}F$ is a non-trivial {\cpt}. Then, 
	\begin{enumerate}
		\item \label{item:const-curv-1} Critical points of $\sigma$ are isolated.
		\smallskip
		\item \label{item:const-curv-2}
		Near each critical point $p$, the level sets of the forward or backward distance function from $p$ (i.e., forward or backward geodesic spheres) coincide with level sets of $\sigma$. Therefore, we can take $u$ to be ($\pm$) the forward or backward distance function from $p$, $r = \dist_p$; i.e., $r = |u|$. This in particular means the forward or backward distance from $p$ coincide (at least near $p$). 
		\smallskip
		\item \label{item:const-curv-3} $\sigma_{u}(0) := \lim_{r \downarrow 0} \sigma_u(u)$ exists and is zero. 
		\smallskip
		\item \label{item:const-curv-4} On a deleted neighborhood of $p$, the ideal metric $g_{\sf ideal}$ has constant flag curvature.
		\smallskip
		\item \label{item:const-curv-5} Setting $\sigma_{uu}(0) := \lim_{r\downarrow 0} \sigma_{uu}(u)$ if it exists; we have $\sigma_{uu}(0)$ is well-defined and non-zero. Therefore, the point $p$ is a non-degenerate critical point of $\sigma$; thus, is a local extrema. 
	\end{enumerate}
\end{lemma}
\begin{proof}[\footnotesize \textbf{Proof}]
		\hfill
		\begin{enumerate}
			\item []{\bf{\footnotesize \textsf{Proof of (\ref{item:const-curv-1})}.}}
			This has been proven in \cite[Theorem 4.2]{FL}. 
			\smallskip
			
			\item []{\bf{\footnotesize \textsf{Proof of (\ref{item:const-curv-2})}.}}
			The argument for this part is similar to the proof of \cite[Theorem 4.2]{FL}.
			\par By \hyperref[item:const-curv-1]{item (\ref{item:const-curv-1})} above, we know $p$ is isolated. Let $o$ be a regular point within both the forward and backward inejctivity radius of $p$. Let $s \ge 0$ and $\overline{po}_{\sf f}(s)$ denote the unique (forward) unit-speed forward geodesic from $p$ to $o$ and $\overline{po}_{\sf b}(s)$ the unique backward unit-speed backward geodesic from $p$ to $o$ (which is a forward unit speed forward geodesic from $o$ to $p$ and must be unique since $o$ is within the backward injectivity radius of $p$).
			\par Denote by $\Sigma^{\sf f}_s$ the connected component of the regular hypersurface $\sigma^{-1}(\sigma(\overline{po}_{\sf f}(s)))$ that contains $\overline{po}_{\sf f}(s)$; thus in particular, $\Sigma^{\sf f}_s$ is a separating (two-sided) hypersurface. Similarly define $\Sigma^{\sf b}_s$.   
			\par Let $z_{\sf f}$ denote a forward nearest point on $\Sigma_o$ form $p$ and $z_{\sf b}$, a backward nearest point on $\Sigma_o$. Denote by $\eta_{\sf f}$ the unique geodesic from $p$ to $z_{\sf f}$ and $\eta_{\sf b}$ the unique backward geodesic from $p$ to $z_{\sf b}$ (if backward complete). Parametrize $\eta_{\sf f}$ by the criterion $\eta_{\sf f}(s) \in \Sigma_s$ and similarly parametrize $\eta_{\sf b}$.   
			\par By the Gauss lemma (e.g. \cite[Theorem 3.18]{Oh}) , $\eta_{\sf f}(s)$ is perpendicular to the level hypersurfaces $\Sigma_s$ w.r.t. $g^{{\sf n}}$ and $\eta_{\sf b}(s)$ is perpendicular to $\Sigma_s$ w.r.t. $g^{-{\sf n}}$.
			\par According to \secref{sec:highlights}\thinspace -\,\hyperref[item:highlight-1]{item (\ref{item:highlight-1})}, perpendicularity to the levels $\Sigma_s$ is independent of $\y$ thus we deduce both $\eta_{\sf f}(s)$ and $\eta_{\sf b}(s)$ must then be perpendicular to $\Sigma_s$ w.r.t. $g^{\sf n}$. Also by \cite[Theorem 3.4]{FL}, we know normal geodesics are reversible and $F({\sf n}) = F(-{\sf n})$ thus both $\eta_{\sf f}(s)$ and $\eta_{\sf b}(s)$ are reversible normal geodesics. 
			\par Now the same argument as in \cite[Theorem 4.2]{FL} using the warped product decomposition, yileds $\eta_{\sf f}(0) = \eta_{\sf b}(0) = p$. This means the parameter $s$ is both the forward and backward distance from $p$. Hence, inparticular, $\Sigma_s$ must be completely within the injectivity radius of $p$ and is diffeomorphic to $\Sp^{n-1}$. 
			\par Since $\Sigma_s$ is within the inejctivity radius of $p$, thus we conclude the level hypersurface $\sigma^{-1}(\sigma(o))$ does not have any other connected component within the inejctivity radius of $p$. In particular, at least within the injectivity radius of $p$, the $\sigma$-levels coincide with the distance spheres (note we observed that the forward and backward distance spheres from $p$ coinicde by the above argument). 
			\par Thus without ambiguity, we can set $\dist_p = r = |u|$.

			 \item []{\bf{\footnotesize \textsf{Proof of (\ref{item:const-curv-3})}.}}
			At $p$ one has $d \sigma(p) = 0$ or $\grad^{\y} \sigma = 0$ for all $\y$; as a reuslt, $\sigma_u(0) := \lim_{r\downarrow 0} \sigma_u(u)$ vanishes by continuity and noting $\grad^{\y} \sigma = \sigma_u\partial_u$ on the regular domain.
			\smallskip 
			\item []{\bf{\footnotesize \textsf{Proofs of (\ref{item:const-curv-4}) and (\ref{item:const-curv-5}) }.}}
			Recall the block form of $(g_{ij})$ from \secref{sec:highlights}\,-\,\hyperref[item:highlight-1]{item (\ref{item:highlight-1})}. Below, we use \enquote{$\bigcdot$} as a placeholder for $\alpha,\beta ,\cdots$. 
				\par From \hyperref[lem:curvature]{Lemma~\ref{lem:curvature}}, we know that the only non-zero coefficients among of $\Riem^{\;\;i}_{h\;\;jk}$ are ${\Riem}^{\;\;{\bigcdot}}_{\,{\bigcdot}\;\;{\bigcdot}{\bigcdot}}$, $\Riem^{\;\,1}_{\,{\bigcdot}\;\;1{\bigcdot}}$ and $\Riem^{\;\,\bigcdot}_{\,1\;\;1{\bigcdot}}$ along with ones that obtained from these three by curvature symmetries. 
				\par Therefore,
				\begin{align*}
					\left\|\Riem\right\|^2_g &= \Riem^{\;\;i}_{h\;\;jk}\Riem^{h\;\;jk}_{\;\;i} \\
					&= \left\|{\Riem}^{\;\;{\bigcdot}}_{\,{\bigcdot}\;\;{\bigcdot}{\bigcdot}}\right\|_{g_{\sf level}}^2 + 2\left\|\Riem^{\;\,{\bigcdot}}_{1\;\,1{\bigcdot}}\right\|^2_{g_{\sf level}} + 2\left\|\Riem^{\;\,1}_{\,{\bigcdot}\;\;1{\bigcdot}}\right\|^2_{g_{\sf level}}\\
					&= \left\|{\Riem}\right\|_{g_{\sf level}}^2  + 2\K^2\left\|\delta^{\bigcdot}_{\bigcdot}\right\|^2_{g_{\sf level}} + 2\K^2 \left\|g_{\bigcdot\bigcdot}\right\|^2_{g_{\sf level}}\\
					&= \left(\rho'\right)^{-4}(r)\left\|{\Riem}\right\|_{g_{\sf ideal}}^2  + 4(n-1)^2\K^2\\
					&= \left(\rho_u\right)^{-4}(r)\left\| {^{\sf level}\Riem}^{\;\;\alpha}_{\beta\;\;\gamma\kappa} -\left(\rho_{uu}(r)\right)^2\left( g^{\sf ideal}_{\beta \kappa}\delta^\alpha_\gamma - g^{\sf ideal}_{\beta\gamma}\delta^\alpha_\kappa \right) \right\|_{g_{\sf ideal}}^2 + 4(n-1)^2\K^2.
				\end{align*}
				\par First, note that since $\overstar{\Gamma}$ and $G^i$ are invariant under constant rescaling of the Finslerian metric, we deduce that
				\begin{align*}
					{^{\sf level}\Riem}^{\;\;\alpha}_{\beta\;\;\gamma\kappa} = {^{\sf ideal}\Riem}^{\;\;\alpha}_{\beta\;\;\gamma\kappa}.
				\end{align*}	
				\par Now letting $r \downarrow 0$, the LHS is bounded; by virtue of the fact that ${^{\sf level}\Riem}$ is independent of $u$, we deduce that
				\begin{align}\label{eq:level-curv}
					{^{\sf level}\Riem}^{\;\;\alpha}_{\beta\;\;\gamma\kappa} &= \lim_{r\downarrow 0}  \left(\rho_{uu}(r)\right)^2\left( g^{\sf ideal}_{\beta \kappa}\delta^\alpha_\gamma - g^{\sf ideal}_{\beta\gamma}\delta^\alpha_\kappa \right).
				\end{align}
				as well as $\lim_{r\downarrow 0} \rho_{uuu}(u) =:\rho_{uuu}(0) = 0$. 
				\par From \eqref{eq:level-curv} and \hyperref[item:const-curv-2]{item (\ref{item:const-curv-2})} above, one infers that
				\begin{enumerate}
					\smallskip
					\item $\rho_{uu}(0) := \lim_{r\downarrow 0} \rho_{uu}(u)$ (or equivalently $\sigma_{uu}(0)$) exists.
					\smallskip
					\item $g_{\sf ideal}$ is a Finslerian metric on the $(n-1)$-sphere with constant flag curvature $(\rho_{uu}(0))^2$.  
					\smallskip
					\item  $\rho_{uu}(0) \neq 0$; since by the Finslerian Cartan-Hadamard theorem, there cannot exist non-positively flag curved metrics on simply connected closed manifolds; e.g. see~\cite[Theorem 8.6]{Oh}
				\end{enumerate} 
		\end{enumerate}
\end{proof}
\subsection*{Proof of the \hyperref[thrm:main]{Theorem}}
\par By \hyperref[lem:level-const-curv]{Lemma~\ref{lem:level-const-curv}}, we know that $\left(\Sigma, g_{\sf level} \right)$ has constant positive flag curvature. 

\par If $F$ is in addition reversible, then clearly, so is the induced Finslerian norm on $\Sigma$. Therefore, by the rigidity result of Kim-Min~\cite{KM} (``reversible Finslerian manifolds with constant positive flag curvature are Riemannian''), we deduce that the level hypersurfaces must be Riemannian round spheres. In particular, 
\begin{align*}
	{^{\sf level} C}_{\alpha\beta\gamma}(\vecv) \equiv 0.
\end{align*}
\par Now by \hyperref[lem:curvature]{Lemma~\ref{lem:curvature}}\,-\,\hyperref[item:curvature-1]{item (\ref{item:curvature-1})}  and \secref{sec:highlights}\,-\,\hyperref[item:highlight-4]{item (\ref{item:highlight-4})}, we deduce that $C_{ijk} \equiv 0$ in the coordinate system mentioned in \secref{sec:highlights}, thus on the regular domain, $\mathbf{C} \equiv 0$. 
\par Since by \hyperref[lem:level-const-curv]{Lemma~\ref{lem:level-const-curv}}\,-\,\hyperref[item:const-curv-1]{item (\ref{item:const-curv-1})} (see~\cite[Theorem 4.2]{FL} for the proof), the set of critical points is a discrete set, by continuity, we deduce that $\mathbf{C} \equiv 0$ on $M$. Therefore, $(M,F)$ is Riemannian. 
\par $(M,F)$ is Berwaldian is equivalent to $\overstar{\Gamma}^i_{jk}$ being independent of $\y$; see~\cite[Theorem 10.2.1]{BCS}. Therefore, based on \hyperref[lem:curvature]{Lemma~\ref{lem:curvature}}\,-\,\hyperref[item:curvature-1]{item (\ref{item:curvature-1})}, if $(M,F)$ is Berwaldian then so is an ideal hypersurface $\left(\Sigma, g_{\sf ideal} \right)$. 
\par Thus in this case, an ideal hypersurface is of constant positive flag curvature and is Berwaldian. Similar to the previous case, we only need to argue that such a Berwaldian space is Riemannian.  
\par The fact that a Berwald manifold with a constant positive flag curvature must be Riemannian has been established in~\cite{BHS} using the Szabo's famous classification of Berwaldian structures~\cite{Szabo}; and in the recent article \cite{TN} by showing that in this case, vanishing $S$-curvature and mean Landsberg curvature implies vanishing mean Cartan tensor. The conclusion then follows from Diecke's theorem. 
\par This concludes the proof of the \hyperref[thrm:main]{Theorem}.
\newline\phantom{}\hfill\scalebox{1.2}{\qed}
\subsection*{Proof of the \hyperref[cor:main]{Corollary}}
\par Upon applying the Riemannian theory, one concludes $\sigma$ has at most twocritical points. The classification also follows from the Riemannian classification; see the beginning of~\secref{sec:intro}.
 \newline\phantom{}\hfill\scalebox{1.2}{\qed}

\vspace{10mm}

\begin{thebibliography}{99}
\bibitem{AZbook} 
	Akbar-Zadeh, H., 
	\textit{Initiation to global Finslerian geometry}, 
Elsevier Science, 2006. 
\bibitem{AIM}
Antonelli, P.L. and Ingarden, R.S. and Matsumoto, M.,
\textit{The Theory of Sprays and
	Finsler Spaces with Applications in Physics and Biology}, Kluwer Academic Publishers, 1993.
\bibitem{BCS}
Bao, D. and Chern, S. S. and Shen, Z.,
\textit{An introduction to Riemann-Finsler geometry}, Springer, 2000.
\bibitem{BF}
Bejancu, A. and Farran, H. R., 
\textit{Geometry of pseudo-Finsler submanifolds}, Kluwer Academic Publishers, 2000.
\bibitem{BHS}
Boonnam, N. and Hama, R. and Sabau, S. V.,
\textit{Berwald space of bounded curvature are Riemannian},
Acta Mathematica Academiae Paedagogicae Nyiregyhaziensis, {\bf 33} (2017), 339--347.
	\bibitem{Br}
Brinkman H. W., \textit{Einstein spaces which are mapped conformally on each other}, Math. Ann. {\bf 94} (1925), 119--145.
\bibitem{FL}
Fathi, Z. and Lakzian, S., 
\textit{The rigidity of conformal circle-preserving transformations on Berwaldian manifolds}, DOI: \url{https://doi.org/10.48550/arXiv.2607.01350}. 


\bibitem{Ish}
Ishihara, S.,
\textit{On infinitesimal concircular transformations},
K\^odai Math. Sem. Rep. {\bf 12} (1960), 45--56.
\bibitem{Ish-Tash}
Ishihara, S. and Tashiro, Y.,
\textit{On Riemannian manifolds admitting a concircular
	transformation},
Math. J. Okayama Univ. {\bf 9} (1959), 19--47.
\bibitem{Kan}
Kanai, M., 
\textit{On a differential equation characterizing a
	Riemannian structure of a manifold},
Tokyo J. Math. {\bf 6}
(1983), 143--151.
\bibitem{KM}
Kim, C.-W. and Min, K.:,
\textit{Finsler metrics with positive constant flag curvature},
Arch. Math. (Basel) {\bf 92} (2009), 70--79.
\bibitem{Knebel}
Knebelman, M . S.,
\textit{Conformal geometry of generalized metric spaces},
Proc. Nat. Acad. Sci. USA {\bf 15} (1929), 376--379.
\bibitem{Kuh}
K\"{u}hnel, W., 
\textit{Conformal Transformations between Einstein Spaces.}
In collection: Conformal Geometry, Kulkarni, R.S., Pinkall, U. (eds), 1988.
\bibitem{Mat1}
Matsumoto, M.,	
\textit{The induced and intrinsic Finsler connections
	of a hypersurface and Finslerien
	projective geometry},
J. Math. Kyoto Univ. (JMKYAZ)
{\bf 25} no. 1 (1985), 107--144.
\bibitem{NM}
Nomizu, K. and Yano, K., \textit{On circles and spheres in Riemannian geometry.} Math.
Ann. {\bf 210} (1974), 163--170. 
\bibitem{Oh}
Ohta, S. I.,
\textit{Comparison Finsler Geometry}, Springer, 2021.
\bibitem{Shen-lectures}
Shen, Z.,
\textit{Lectures on Finsler geometry}, World Scientific Publishing Company, 2001. 
\bibitem{SY} Shen, Z. and  Yang, G.,
\textit{On concircular transformations in Finsler geometry}, Results Math. {\bf 74} no. 162 (2019), 435--445.
\bibitem{Szabo}
Szab\'o, Z. I.,
\textit{Positive definite Berwald spaces. Structure theorems on Berwald spaces},
Tensor (N.S.), {\bf 35} (1981), 25--39.
\bibitem{Tash1}
Tashiro, Y., 
\textit{Complete Riemannian manifolds and some vector
	fields},
Transact. AMS 117 (1965), 251-275.
\bibitem{Tash2}
Tashiro, Y., 
\textit{Conformal transformations in complete Riemannian manifolds},
Publ. Study Group of Geometry {\bf 3} (1967).

\bibitem{TN}
Tayebi, A. and Najafi, B.,
\textit{On Berwald Spaces with non-Zero Flag Curvature},
arXiv:2601.20087.
\bibitem{VO}
Vogel, W. O.,
\textit{Kreistreue transformationen in Riemannschen
	r\"{a}umen}. Arch. der Math. {\bf 21} (1970), 641--645.
\bibitem{Yano1}
Yano, K., \textit{On circular geometry I. Concircular transformations},
Proc. Imp. Acad. Tokyo {\bf 16} (1940), 195--200.
\bibitem{Yano2}
\bysame,
\textit{On circular geometry II. Integrability conditions of $\rho_{\mu \nu} = \varphi g_{\mu\nu}$}
Proc. Imp. Acad. Tokyo {\bf 16} (1940), 354--360. 
\bibitem{Yano3}
\bysame,
\textit{On circular geometry III. Theory of curves},
Proc. Imp. Acad. Tokyo
{\bf 16} (1940), 442--448. 
\bibitem{Yano4}
\bysame,
\textit{On circular geometry IV. Theory of subspaces},
Proc. Imp. Acad. Tokyo {\bf 16} (1940), 505--511.
\bibitem{Yano5}
\bysame,
\textit{On circular geometry V. Einstein spaces},
Proc. Imp. Acad. Tokyo {\bf 18} (1942),
446--451.
\end{thebibliography}
\end{document}